\documentclass[11pt]{article}
\usepackage{amsmath,amsthm,amssymb,amscd}
\usepackage{graphicx}
\usepackage{graphics}
\usepackage{verbatim}

\usepackage{mathrsfs}
\usepackage{color}
\usepackage[top=1in, bottom=1in, left=1.25in, right=1.25in]{geometry}
\usepackage{enumerate}

\usepackage[utf8]{inputenc}
\usepackage[english]{babel}
\usepackage[T1]{fontenc}

\usepackage{cite} 

\usepackage{acronym} 
\usepackage[export]{adjustbox}
\usepackage{enumitem} 

  \usepackage{multicol}
  \usepackage{multirow}
  
\newtheorem*{rep@theorem}{\rep@title}
\newcommand{\newreptheorem}[2]{%
	\newenvironment{rep#1}[1]{%
		\def\rep@title{#2 \ref{##1}}%
		\begin{rep@theorem}}%
		{\end{rep@theorem}}}
\newreptheorem{theorem}{Theorem}

\usepackage{amsthm}
\newtheorem{defi}{Definition}[section]
\newtheorem{theo}[defi]{Theorem}
\newtheorem{conj}[defi]{Conjecture}
\newtheorem*{conj*}{Conjecture}

\numberwithin{subcase}{case}
\newtheorem*{theo*}{Theorem}
\newtheorem{claim}{Claim}
\newtheorem*{claim*}{Claim}
\newtheorem{lem}[defi]{Lemma}

\newtheorem*{satz*}{Satz}

\newtheorem{coro}[defi]{Corollary}

\usepackage{thmtools}
\usepackage{thm-restate}

\usepackage{graphicx}
\usepackage{epic}
\usepackage{tikz} 
\usetikzlibrary{er,positioning}
\usepackage[arrow, matrix, curve]{xy}

\usepackage{tabularx} 
\usepackage{array}

\usepackage{multirow}

\usepackage{color}												
\usepackage{colortbl}	
\usepackage{xcolor}

\usepackage{hyperref}

\usepackage{cleveref}

\usepackage{titlesec}

\DeclareMathOperator{\iso}{iso}
\DeclareMathOperator{\Iso}{Iso}
\DeclareMathOperator{\deficiency}{def}
\DeclareMathOperator{\nc}{nc}
\DeclareMathOperator{\supp}{supp}

\DeclareRobustCommand{\N}{\mathbb{N}} 
\DeclareRobustCommand{\R}{\mathbb{R}} 
\DeclareRobustCommand{\Z}{\mathbb{Z}}

\title{Fractional matchings, component-factors and edge-chromatic critical graphs}
\author{Antje Klopp \thanks{
				Institute for Mathematics,
       	Paderborn University,
       	Warburger Str. 100,
       	33098 Paderborn,
       	Germany; aklopp@math.uni-paderborn.de},
Eckhard Steffen\thanks{
       	Paderborn Center for Advanced Studies and
				Institute for Mathematics,
       	Paderborn University,
       	Warburger Str. 100,
       	33098 Paderborn,
       	Germany;	es@upb.de, ORCID 0000-0002-9808-7401}}

\begin{document}

\clearpage
\setlength{\hoffset}{0mm}
\newpage

\renewcommand{\sectionmark}[1]		
	{\markboth{\sc \thesection{} #1}{}}

\hypersetup{pageanchor=true} 

\date{}

\maketitle

\begin{abstract}
\noindent 
The first part of the paper studies star-cycle factors of graphs. It characterizes star-cycle factors of a graph $G$ 
and proves upper bounds for the minimum number
of $K_{1,2}$-components in a $\{K_{1,1}, K_{1,2}, C_n\colon n\ge 3\}$-factor of a graph $G$. 
Furthermore, it shows where these components 
are located with respect to the Gallai-Edmonds decomposition of $G$ and it characterizes the edges 
which are not contained in any $\{K_{1,1}, K_{1,2}, C_n\colon n\ge 3\}$-factor of $G$. 

\noindent The second part of the paper proves that every edge-chromatic critical graph $G$ 
has a $\{K_{1,1}, K_{1,2}, C_n\colon n\ge 3\}$-factor, and 
the number of $K_{1,2}$-components is bounded in terms of its fractional matching number. Furthermore, it shows that
for every edge $e$ of $G$, there is a $\{K_{1,1}, K_{1,2}, C_n\colon n\ge 3\}$-factor $F$ with $e \in E(F)$. 
Consequences of these results for Vizing's critical graph conjectures are discussed. 
\end{abstract}


%
%
%
%
%
%

\section{Introduction and Motivation}

We consider finite simple graphs. For a graph $G$, $V(G)$ and $E(G)$ denote the set of vertices and the set of edges, respectively. 
For a vertex $v$ of $V(G)$, $E_G(v)$ denotes the set of edges which are incident to $v$. The degree of $v$, denoted by $d_G(v)$, is $|E_G(v)|$.  
The maximum degree of a vertex of $G$ is denoted by $\Delta(G)$ and the minimum degree of a vertex of $G$ is denoted by $\delta(G)$. If $\Delta(G) = \delta(G) = k$, then $G$ is $k$-regular. If $G$ is a 2-regular graph then it is also called a cycle,
and if $G$ is a connected 2-regular graph, then we also call $G$ a circuit.
For $v \in V(G)$, the set of neighbors of $v$ is denoted by $N_G(v)$. Clearly,
$d_G(v) = |E_G(v)| = |N_G(v)|$, for simple graphs. 
For a set $X\subseteq V(G)$, the neighborhood of $X$ is defined as $N_G(X)=\bigcup_{x\in X} N_G(x)$.  For $S \subseteq V(G)$, the set of edges with precisely one end in $S$ is denoted by 
$\partial_G(S)$. For $A, B \subset V(G)$, the set of edges with one end in $A$ and the other in $B$ is denoted by $E_G(A,B)$. Hence,
$E_G(S,V(G)-S) = \partial_G(S)$. If there is no harm of confusion, then we will omit the indices.  

A set $M$ ($M \subseteq E(G)$ or $M \subset V(G)$) is independent, if no two elements of $M$ are adjacent.  An independent set of edges is
also called a matching of $G$. The maximum cardinality of a matching of $G$ is the matching number of $G$, which is denoted by $\mu(G)$.  A
matching $M$ with $|M|=\mu(G)$ is a maximum matching of $G$. The number of vertices which are not incident to an edge of a maximum matching
is the matching-deficiency of $G$, and it is denoted by $\deficiency(G)$. Clearly, $\deficiency(G) = |V(G)|- 2\mu(G)$.

A fractional matching of $G$ is a function $f: E(G) \rightarrow [0,1]$ such that $\sum_{e \in E_G(v)}f(e) \leq 1$ for all $v \in V(G)$. If $f(e) \in \{0,1\}$
for each edge, then $f$ is the characteristic function of a matching of $G$. The fractional matching number $\mu_f(G)$ is 
$\sup\{\sum_{e \in E(G)}f(e): f \text{ is a fractional matching of } G\}$. Clearly, $\mu_f(G) \leq \frac{1}{2}|V(G)|$ and if 
$\mu_f(G) = \frac{1}{2}|V(G)|$, then $f$ is a fractional perfect matching. For a fractional matching $f$ 
the set $\{e: e \in E(G) \text{ and } f(e) \not = 0\}$ is the support of $f$ and it is denoted by $\supp(f)$.  

\begin{theo} [\cite{Scheinerman_Buch} (Theorem 2.1.5)] \label{max_frac_matching} For any graph $G$, $2 \mu_f(G)$ is an integer. 
Moreover, there is a fractional matching $f$ for which 
$\sum_{e \in E(G)}f(e) = \mu_f(G)$ and $f(e) \in \{0, \frac{1}{2}, 1\}$ for every $e \in E(G)$. 
\end{theo}

Let $G$ be a graph and $g,f:V(G)\to \Z$ be two functions such that $0 \leq g(v)\le f(v)$ for all $v\in V(G)$.
A $(g,f)$-factor is a spanning subgraph $F$ of $G$ that satisfies 
$g(v)\le d_F(v)\le f(v) \text{ for all } v\in V(G)$. If $g(v) = a$ and $f(v)=b$ for all $v \in V(G)$, then $F$ is a $[a,b]$-factor, and if $a=b=k$,
then $F$ is a $k$-factor of $G$. Clearly, if $F$ is a 1-factor, then $E(F)$ is a perfect matching of $G$. If $F$ is a factor of a graph $G$, then a path is $F$-alternating, if its edges are in $F$ and $E(G)-F$ alternately. 

For a set $\cal{S}$ of connected graphs, a spanning subgraph $F$ of $G$ is called an $\cal{S}$-factor if each component of $F$ is isomorphic to an element of $\cal{S}$. 
If $H \in \cal{S}$, then a component of $F$ which is isomorphic to $H$ is called an $H$-component of $F$. 
A component is trivial if it consists of a single vertex and non-trivial otherwise. 
The set of trivial components of $G$ is denoted by $\Iso(G)$ and $\iso(G)$ denotes $|\Iso(G)|$.

The complete bipartite graph with bipartition $(A,B)$ and $|A|=r$, $|B|=s$ is denoted by $K_{r,s}$. 
In case of $r=1$,  $K_{1,s}$ is called a star and the vertex of degree $s$ is its center vertex. 
For $K_{1,1}$, either of the two vertices can be regarded as its center vertex. A $\{K_{1,1},\dots, K_{1,t}, C_m :  m \geq 3 \}$-factor of $G$ is called a star-cycle factor. 

For a set $S$ of vertices let $G[S]$ and $G-S$ be the subgraphs of $G$ induced by $S$ and $V(G)-S$, respectively. The following theorems
characterize some component factors of graphs. 

\begin{theo} [\cite{T1953}] \label{Tutte_53} A graph $G$ has a $\{K_{1,1}, C_m : m\geq 3\}$-factor if and only if $\iso(G-S) \leq |S|$ for all $S \subseteq V(G)$. 
\end{theo}

In terms of fractional perfect matchings, Theorem \ref{Tutte_53} is equivalent to the following formulation.

\begin{theo}[\cite{Scheinerman_Buch}] \label{Scheinerman}
 A graph $G$ has a fractional perfect matching if and only if  $\iso(G-S) \leq |S|$ for all $S \subseteq V(G)$. 
\end{theo}

The following theorems characterize graphs which satisfy relaxed conditions.

\begin{theo}[\cite{AE1980}]\label{theorem_Akiyama_Era}
 A graph $G$ has a $\{K_{1,1}, K_{1,2}, C_m\colon m\ge 3\}$-factor if and only if $\iso(G-S)  \le 2 \vert S \vert$ for all $S \subseteq V(G)$.
\end{theo}

\begin{theo} [\cite{LasVergnas1978, Amahashi1982}] \label{K_1n_factor} Let $n \geq 2$ be an integer. 
A graph $G$ has a $\{K_{1,1}, \dots, K_{1,n}\}$-factor if and only if $\iso(G-S) \leq n|S|$ for all $S \subseteq V(G)$. 
\end{theo}

These results had been generalized by Berge and Las Vergnas \cite{Berge_star_cycle_factor1978} to star-cycle factors.

\begin{theo} [\cite{Berge_star_cycle_factor1978}] \label{star_cycle_factor} Let $G$ be a graph and $f: V(G) \rightarrow \{1,2,3\dots\}$ 
be a function, and let $W= \{v: v\in V(G) \text{ and } f(v)=1\}$. The graph $G$ has a star-cycle factor $F$ such that

$(i)$ $d_F(v) \leq f(v)$ if $v$ is the center vertex of a star component of $F$, and 

$(ii)$ $V(C) \subseteq W$ for each circuit component $C$ of $F$

if and only if $\iso(G-S) \leq \sum_{v \in S}f(v)$ for all $S\subseteq V(G)$. 
\end{theo}

For each finite graph $G$, if $\iso(G) =0$, then there is an integer $n$ such that $\iso(G-S) \leq n|S|$ for all $S \subseteq V(G)$. 
Consequently, the following statement is true. 
 
\begin{coro} \label{exist_star_cycle_factor}
Every graph without trivial components has a star-cycle factor. 
\end{coro}

The paper is organized as follows. Section \ref{Factors} studies general graphs
while Section \ref{Edge-colorings} studies edge-chromatic critical graphs.  
The edge-chromatic number $\chi'(G)$ of a graph $G$ is the
minimum number $k$ of matchings which are needed to cover the edge set of $G$.
In 1965, Vizing \cite{V1965} proved that $\chi'(G) \in \{\Delta(G), \Delta(G) + 1\}$ for a graph $G$.
For $k \geq 2$, a graph $G$ is $k$-critical, if $\Delta(G)=k$, $\chi'(G) = k+1$ and $\chi'(H) \leq k$ for each proper subgraph $H$ of $G$.
We often say that $G$ is a critical graph, if there is a $k$, such that $G$ is a $k$-critical graph.

In Section \ref{Factors} we characterize graphs with specific star-cycle factors in terms of their fractional matching number.
In particular, we give an upper bound for the size of a star and for the number of star components which are different from $K_{1,1}$,
and we locate the star components of a factor with respect to the Gallai-Edmonds decomposition of $G$. We further address
the question for which $e \in E(G)$ there is a specific star-cycle factor $F$ with $e \in E(F)$. 

In addition to these statements, the following theorems are the main results of this section regarding the application to questions on factors of edge-chromatic critical graphs. 
Let $\min(G, K_{1,2})$ denote the minimum number of $K_{1,2}$-components in a
$\{K_{1,1}, K_{1,2}, C_m : m \ge 3\}$-factor of $G$.

\begin{reptheorem}{kez_condition}
	If a graph $G$ has a $\{K_{1,1}, K_{1,2}, C_m : m \ge 3\}$-factor,
	then $\mu_f(G) = \frac{1}{2}(|V(G)|-\min(G, K_{1,2}))$. 
\end{reptheorem}

\begin{reptheorem}{Lemma_condition_not_factor}
	Let $G$ be a graph that has a $\{K_{1,1}, K_{1,2}, C_m\colon m\ge 3\}$-factor. 
	For $e \in E(G)$, say $e=uv$, there is no $\{K_{1,1}, K_{1,2}, C_m\colon m\ge 3\}$-factor which contains $e$ 
	if and only if there is a subset $S$ of $V(G)$ that satisfies
	\begin{enumerate}
		\item [$(i)$] $u,v\in S$
		\item [$(ii)$] $2\vert S\vert -2\le \iso(G-S) \le 2\vert S\vert$.
	\end{enumerate}
	Furthermore, the inequalities of $(ii)$ are tight. 
\end{reptheorem}

In Section \ref{Edge-colorings}   we prove
that every edge chromatic critical graph has $\{K_{1,1}, K_{1,2}, C_m\colon m\ge 3\}$-factor.
The following two theorems are the main results of the paper. The maximum cardinality of an independent set of vertices is the independence number of $G$ which is denoted by $\alpha(G)$.

\begin{reptheorem}{first_thm}
	Let $G$ be a critical graph. Then $G$ has a $\{K_{1,1}, K_{1,2}, C_m\colon m\ge 3\}$-factor with 
	$\min(G,K_{1,2}) = |V(G)| - 2\mu_f(G)$. In particular,
	$\min(G, K_{1,2}) \le \frac{1}{5} \vert V(G)\vert$ and  $\alpha(G)\le \frac{3}{5}\vert V(G)\vert$ for all $\Delta(G)\ge2$.
\end{reptheorem}

The statement $\alpha(G)\le \frac{3}{5}\vert V(G)\vert$ for all critical graphs 
was first proved by Woodall \cite{W2011_1}.

\begin{reptheorem}{main_new}
	Let $G$ be a critical graph. For every edge $e$ there is a $\{K_{1,1}, K_{1,2}, C_m\colon m\ge 3\}$-factor 
	$F$ with $e\in E(F)$.
\end{reptheorem}

These results have some consequences for Vizing's critical graph conjectures, see \cite{Cao_etal_survey_2019}. 

\begin{conj} [\cite{V1968}] \label{2FC}
If $G$ is a critical graph, then $G$ has a 2-factor.
\end{conj}

\begin{conj} [\cite{V1965_2}] \label{INC}
If $G$ is a critical graph, then  $\alpha(G) \leq \frac{1}{2}|V(G)|$.
\end{conj}

Both conjectures are open for a long time and our results on star-cycle-factors can be seen as an approximation.
Figure \ref{summary_1} shows the connection between these conjectures, fractional matchings, component-factors and there applications on critical graphs. 
The paper closes with the study of fractional matchings on critical graphs.

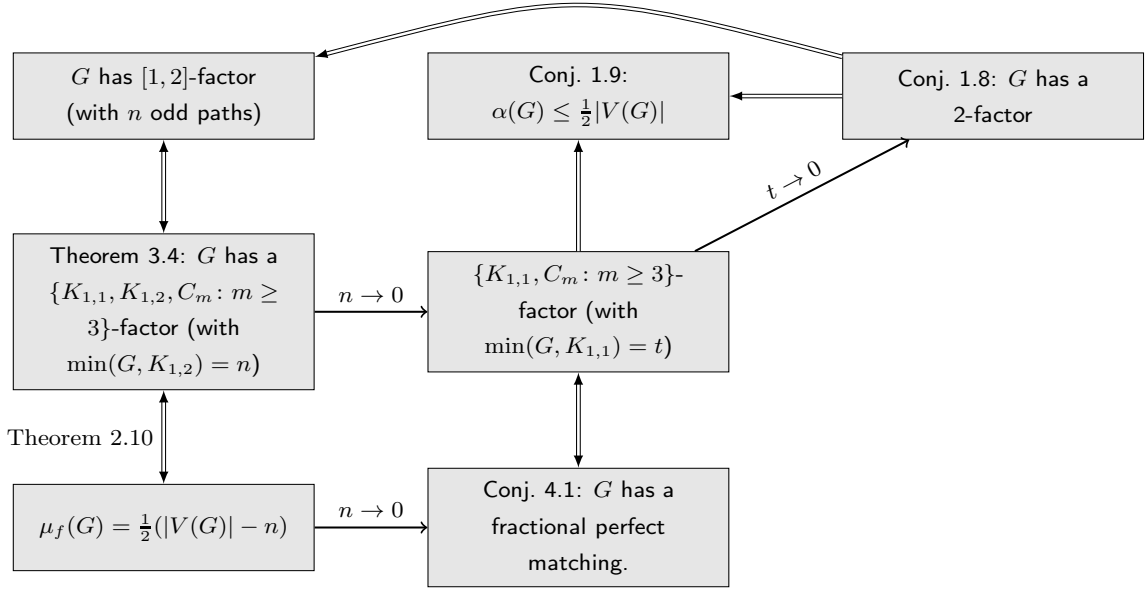
\begin{figure}[h]

\centering
\resizebox{15.5cm}{8cm}{
\begin{tikzpicture}
[node distance = 1cm, auto,font=\footnotesize,
every node/.style={node distance=3cm},
comment/.style={rectangle, inner sep= 5pt, text width=4cm, node distance=0.25cm, font=\scriptsize\sffamily},
force/.style={rectangle, draw, fill=black!10,minimum size=4cm, inner sep=5pt, text width=3.5cm, text badly centered, minimum height=1.2cm
, font=\footnotesize\sffamily}
] 

\node [force] (INC2) { $\{K_{1,1}, C_m \colon m\ge 3\}$-factor (with $\min(G, K_{1,1})= t$)};
\node [force, above of=INC2] (INC1) {Conj.~\ref{INC}: $\alpha(G) \leq \frac{1}{2}|V(G)|$};
\node [force, text width=3cm,left=1.5cm of INC1] (OTF1) {$G$ has $[1,2]$-factor (with $n$ 
	odd paths)};
\node [force, text width=3cm,right=1.5cm of INC1] (TF1) {Conj.~\ref{2FC}: $G$ has a 2-factor};
\node [force, left=1.5cm of INC2] (OTF2) {Theorem~\ref{first_thm}: $G$ has a $\{K_{1,1}, K_{1,2}, C_m\colon m\ge 3\}$-factor (with $\min(G, K_{1,2})= n$)};
\node [force, below of=INC2] (INC3) {Conj.~\ref{FPMC}: $G$ has a fractional perfect matching.};
\node [force, below of=OTF2] (OTF3) {$\mu_f(G)=\frac{1}{2}(\vert V(G)\vert -n)$};

\draw[double,double distance=1pt,<->,>=latex]  (OTF1) -- (OTF2) node[pos=.5,sloped,above] {};
\draw[double,double distance=1pt,<->,>=latex] (INC2) -- (INC3) node[pos=.5,sloped,above] {};

\draw[double,double distance=1pt,->,>=latex] (TF1) ..controls (0,4.5) ..(OTF1) node[pos=.5, left] {};
\draw[double,double distance=1pt,->,>=latex] (INC2) -- (INC1) node[pos=.5,sloped,above] {};
\draw[->, thick] (OTF2) -- (INC2) node[pos=.5,sloped,above] {$n\to 0$};
\draw[->, thick] (INC2) -- (TF1) node[pos=.5,sloped,above] {$t\to 0$};
\draw[->, thick] (OTF3) -- (INC3) node[pos=.5,sloped,above] {$n\to 0$};
\draw[double,double distance=1pt,->,>=latex] (TF1) -- (INC1) node[pos=.5,sloped,above] {};
\draw[double,double distance=1pt,<->,>=latex] (OTF2) -- (OTF3) node[pos=.5,left] {Theorem  \ref{kez_condition}};
;

\end{tikzpicture} }
\caption{Conjectures and Theorems for edge-chromatic critical graphs }\label{summary_1}
\end{figure}


\section{Fractional matching number and star-cycle factors} \label{Factors}

A graph $G$ is factor-critical if $G-v$ has a perfect matching for each $v \in V(G)$. Analogously, a matching is near perfect if it covers all vertices but one. 
Let $D(G)$ be the set of vertices of $G$ which are missed by at least one 
maximum matching of $G$, let $A(G) = N(D(G)) - D(G)$ and $C(G) = V(G) - (D(G) \cup A(G))$. We call the triple $(D(G),A(G),C(G))$ a Gallai-Edmonds decomposition of $G$. If there is no harm of confusion we shortly write $(D,A,C)$ instead of $(D(G),A(G),C(G))$.
We will use the fundamental Gallai-Edmonds structure theorem.

\begin{theo}[\cite{Edmonds1965, Gallai1963}] \label{GE-decomposition}
Let $G$ be a graph. If $(D,A,C)$ is a Gallai-Edmonds decomposition of $G$, then

\begin{enumerate}
\item every component of $G[D]$ is factor-critical,
\item $G[C]$ has a perfect matching,
\item every maximum matching consists of a near perfect matching on each component of $G[D]$, a perfect matching on $G[C]$, and a matching which matches every vertex of $A$ to one distinct component of $G[D]$, and
\item $\mu(G) = \frac{1}{2}(|V(G)| - c(G[D]) + |A|)$, where $c(G[D])$ is the number of components of $G[D]$.
\end{enumerate}
\end{theo}

Next we formulate a sharpening of this result in the context of fractional matchings. Let $M$ be a maximum matching of a graph $G$
and $\nc(M)$ be the number of non-trivial components of $G[D]$ that are not matched by an edge 
$e \in M \cap E(D,A)$, and $\nc(G) = \max\{\nc(M) : M \text{ is a maximum matching of } G\}$. 

\begin{theo}\label{LiuLiu}Let $G$ be a graph and $n \geq 0$ be an integer. If $\mu_f(G) = \frac{1}{2}(|V(G)| - n)$, then
\begin{enumerate}
 \item $n = \deficiency(G) - \nc(G)$ \cite{Liu_Liu2002}, \label{LiuLiu_1}
 \item $n = \max\{\iso(G-S) - |S| : S \subseteq V(G)\}$ \cite[Theorem 2.2.6]{Scheinerman_Buch}.
\end{enumerate}
\end{theo}

%
We call a set $S$ with $\iso(G-S) = |S| + n$ a witness for $\mu_f(G)$.
A crucial point in the proof of Theorem \ref{LiuLiu}(\ref{LiuLiu_1}) is that every non-trivial component of $G[D]$ has a fractional perfect matching. The following theorem shows that they have even more structural properties.

\begin{theo}[\cite{Cornuejols_Pulleyblank_1983}]\label{Pulleyblank}
Let $G$ be a factor-critical graph with $\vert V(G)\vert >1$. Then $G$ has a fractional perfect  matching $f$ with $f(e)\in \{0,\frac{1}{2}, 1\}$ for every $e\in E(G)$ and the set $\{e \colon e\in E(G) \text{ and } f(e)=\frac{1}{2}\}$ forms exactly one odd circuit.
\end{theo}

Furthermore, every maximum matching of a graph $G$ is contained 
in the support of a fractional matching with values in $\{0, \frac{1}{2}, 1\}$.
Let $M$ be a maximum matching with $\nc(M) = \nc(G)$. A maximum fractional matching $f$ with $M \subseteq \supp(f)$
is called a canonical maximum fractional matching of $G$ (with respect to $M$). 
  
Theorem \ref{LiuLiu} shows that every graph has a canonical maximum fractional matching. 
A look into the proof details of Theorem \ref{LiuLiu}(\ref{LiuLiu_1}) yields that it is also shown that $A(G)$ contains a witness for $\mu_f(G)$.
We will state this fact in a more detailed manner in the following corollary. 

\begin{coro}
Let $G$ be a graph, $n \geq 0$ be an integer, $M$ be a maximum matching of $G$ and $\mu_f(G) = \frac{1}{2}(|V(G)| - n)$. If $f$ is a canonical maximum fractional matching
w.r.t. $M$, then $\Iso(G[D])$ contains two disjoint subsets $D^+$ and $D^-$ with

\begin{enumerate} 
\item $D^- = \{v \in \Iso(G[D]): v \text{ is not matched by } M\}$ and $|D^-| = n$,
\item $D^+ = \{w \in \Iso(G[D]):\text{there is an $M$-alternating path from $w$ to some vertex of } D^-\}$,
\item $M$ induces a perfect matching on $D^+ \cup N(D^+ \cup D^-)$; in particular, $|N(D^+ \cup D^-)| = |D^+|$, and
\item $N(D^+ \cup D^-)$ is a witness for $\mu_f(G)$.
\end{enumerate} 
\end{coro}

If $F$ is a star-cycle factor of $G$, then let $t_i^F$ denote the number of $K_{1,i}$-components of $F$ and let
$l(G) = \min\{\sum_{i=1}^\infty (i-1)t_i^F : F \text{ is a star-cycle factor of } G\}$. 
The next theorem gives a detailed insight into the structure of graphs with respect to their fractional
matching number.  

\begin{theo}\label{star_cycle_factor_neu}
Let $G$ be a connected graph, $n \geq 0$ be an integer and $\lambda$ be the minimum integer such that 
$\iso(G-S) \leq \lambda|S|$ for all $S \subseteq V(G)$. 
If $\mu_f(G) = \frac{1}{2}(|V(G)|-n)$,
then $ \lambda \leq \lceil\frac{n}{\delta(G)}\rceil + 1$ and $G$ has a $\{K_{1,1}, \dots, K_{1,\lambda}, C_m : m \geq 3 \}$-factor $F$, such that $l(G) = \sum_{i=1}^\lambda (i-1)t_i^F = n$. 
Furthermore, the $K_{1,j}$-components are induced subgraphs of $G$, and for $j \geq 2$,
their center vertices are in $N(D^+ \cup D^-) (\subseteq A)$ 
and their leaves are in $D^+ \cup D^-$.
\end{theo}

\begin{proof}
Let $f$ be a canonical maximum fractional matching w.r.t.~$M$. 
For $n=0$ we have $D^-=\emptyset$ and for $n\geq 1$ let $D^- = \{d_1, \dots,d_n\}$.
Let $V_0 = V(G) - D^-$, and for $k \in \{1,\dots,n\}$ let $V_k = V_0 \cup \{d_1, \dots,d_k\}$. Further let $G_k = G[V_k]$, for $k \in \{0,\dots,n\}$. Clearly,
$G_k$ is a subgraph of $G$ and $f$ is a canonical maximum fractional matching of $G_k$ w.r.t.~$M$ with $\mu_f(G_k) = \frac{1}{2}(|V(G_k)|-k)$.

We construct a sequence of subgraphs $F_0, \dots, F_n$ of $G$, 
where the subgraph $F_k$ is the desired $\{K_{1,1}, \dots, K_{1,t_k}, C_m : m \geq 3 \}$-factor on $G_k$, with $t_k \leq	 \lambda$, $l(G_k)=k$ and $G_n=G$. 

If $k=0$, then $G[V_0]$ has a perfect fractional matching, $\iso(G_0-S)\le \vert S\vert$ for all $S\subseteq V(G_0)$ by Theorem \ref{Scheinerman} and the statement follows with Theorem \ref{Tutte_53}, that is, $t^{F_0}_i = 0$ for each $i \geq 2$
and therefore, $l(G_0) = 0$ and $t_0 = 1 \le \lambda$. 

Suppose that $F_k$ has been constructed in $G_k$ for $k$, with $k\leq n-1$. We will construct $F_{k+1}$ in $G_{k+1}$. 

\underline{Case A:} There is a vertex $a \in N_G(d_{k+1})$ with $a \not \in N(\{d_1, \dots,d_k\})$ or $d_{F_k}(a) < \lambda$. 
Then $F_k \cup \{d_{k+1}a\}$ is a $\{K_{1,1}, \dots, K_{1,t_{k+1}}, C_m : m \geq 3\}$-factor of $G_{k+1}$.
The factor $F_{k+1}$ is obtained from $F_k$ by extending a $K_{1,j}$-component, with $j<\lambda$, to a $K_{1,j+1}$-component.
Hence, $t_j^{F_k} - 1 = t_j^{F_{k+1}}$ and $t_{j+1}^{F_k} + 1 = t_{j+1}^{F_{k+1}}$. 
Furthermore, $t_{k+1} \leq t_k + 1 \leq \lambda$.
Thus,
$l(G_{k+1})=\sum_{i=1}^{\lambda} (i-1)t_i^{F_{k+1}} = \sum_{i=1}^{\lambda} (i-1)t_i^{F_{k}} - (j-1)+j = k+1$.

\underline{Case B:} For all $a \in N_G(d_{k+1})$: $d_{F_k}(a) = \lambda$. 
Let $P$ be the set of all vertices of $A(G)\cup D(G)$ for which there is an $F_k$-alternating path with initial vertex $d_{k+1}$,
$T_D = P \cap D(G)$ and $T_A = P \cap A(G)$. Note that $T_D\subseteq \Iso(G[D])$, since $f$ is a canonical maximum fractional matching w.r.t. $M$ and $M$ is a maximum matching with $\nc(M)=\nc(G)$.

If $d_{F_k}(a) = \lambda$ for all $a \in T_A$, then, by the definition of $T_A$ and $T_D$, it follows that
$T_D$ is a set of isolated vertices in $G-T_A$. But $|T_D| = \lambda |T_A| + 1$, a contradiction to 
the choice of $\lambda$.

Hence, there is a  $a' \in T_A$ with $d_{F_k}(a') < \lambda$. Let $p = d_{k+1},a^1, d^1, \dots, a^t, d^t, a'$ be a minimal
$F_k$-alternating path ($d^i \in D(G)$ and $a^i \in A(G)$) with end vertices $d_{k+1}$ and $a'$. Note that
$d_{F_k}(a^i) = \lambda$, $d_{F_k}(d^i)=1$, $a^id^i \in E(F_k)$ and $d_{k+1}a^1, d^ia^{i+1}, d^ta' \not \in E(F_k)$. 
Let $F_{k+1}$ be obtained from $F_k$ by interchanging the edges of $F_k$ and $E(p) - E(F_k)$ in $p$. Hence,
$F_{k+1}$ is a $\{K_{1,1}, \dots, K_{1,\lambda}, C_m : m \geq 3\}$-factor of $G_{k+1}$. As in Case A it follows
that $\sum_{i=1}^{\lambda} (i-1)t_i^{F_{k+1}} = k+1$ and $t_{k+1} \leq \lambda$.

Let $F=F_n$. Then $F$ is a $\{K_{1,1}, \dots, K_{1,\lambda}, C_m : m \geq 3\}$-factor of $G$ and 
$\sum_{i=1}^{\lambda} (i-1)t_i^{F} = n$. We cannot do better since 
$f':E(G) \rightarrow [0,1]$ with $f'(e) = \frac{1}{i}$ if $e$ is an edge of a $K_{1,i}$-component of $F$, $f'(e) = \frac{1}{2}$,
if $e$ is an edge of a circuit of $F$, and $f'(e) = 0$ otherwise, is a fractional matching of $G$ and 
$\sum_{e \in E(G)} f'(e) = \frac{1}{2}(|V(G)| - n)$. 

It remains to show that $\lambda \leq \lceil \frac{n}{\delta(G)} \rceil+1$. 
Without loss of generality we may assume that $d_G(d_1) \leq \dots \leq d_G(d_n)$.
Let $F=F_n$ be the $\{K_{1,1}, \dots, K_{1,t}, C_m : m \geq 3 \}$-factor
as constructed above and $t = t_n$.
Clearly, $\lambda \leq t$. For $k \leq n-1$, $F_{k+1}$ is obtained from $F_k$ either by applying
the construction of Case A or the construction of Case B. In Case A, vertex $a$ can be chosen
such that $d_{F_k}(a) = \min\{d_{F_k}(x) : x \in N_G(d_{k+1})\}$. Thus,
$t_{k+1} = t_k$ if $d_{F_k}(a) < t_k$ and $t_{k+1}=t_k+1 \leq \lambda$ otherwise. 
In case B, we have $t_k = t_{k+1} (=\lambda)$. Since Case B only applies if Case A does not,
it follows that $t \leq \lceil \frac{n}{\delta(G)} \rceil +1$.
Therefore,
$\iso(G-S) \leq  (\lceil \frac{n}{\delta(G)} \rceil +1)|S|$ for all $S \subseteq V(G)$. Since $\lambda$ is minimum, the statement follows. 
\end{proof}

\begin{coro} \label{sum_up} For each graph $G$, 
$l(G) = \deficiency(G) - \nc(G) = \max\{\iso(G-S) - |S| : S \subseteq V(G)\} = |V(G)|-2\mu_f(G)$ and 
$G$ has a $\{K_{1,1},\dots, K_{1,\lambda}, C_m : m\ge 3\}$-factor with $\nc(G)$ circuits.
\end{coro}

\begin{proof}
 The first statement follows directly from Theorem \ref{LiuLiu} and Theorem \ref{star_cycle_factor_neu}. The second statement follows from Theorem \ref{Pulleyblank} and Theorem \ref{star_cycle_factor_neu}.
\end{proof}

\begin{coro}
Let $G$ be a graph. Then 
 \begin{align*}
\alpha(G)\le \frac{1}{2}\left(\vert V(G)\vert +l(G)-\nc(G)\right).
 \end{align*}
\end{coro}
\begin{proof}
By Corollary \ref{sum_up} $G$ has a  $\{K_{1,1},\dots, K_{1,\lambda}, C_m : m\ge 3\}$-factor  with $\nc(G)$ odd circuits and $l(G)$  vertices extend $K_{1,1}$-components to $K_{1,j}$-components, $1<j\le \lambda$. Therefore, $\alpha(G)\le \frac{1}{2} (\vert V(G)\vert-l(G)) -\frac{1}{2}\nc(G)+l(G)$.
\end{proof}

\begin{theo} \label{fm_edge}
Let $G$ be a graph and $e' \in E(G)$. If there is a maximum fractional matching $f$ of $G$ with $f(e') \not=0$, then there is a
maximum fractional matching $f'$ with $f'(e) \in \{0, \frac{1}{2}, 1\}$ for all $e \in E(G)$ and $f'(e') \not= 0$,
and the components of $\supp(f')$ are $K_{1,1}$'s or odd circuits.  
\end{theo}

\begin{proof} Let $f$ be a maximum fractional matching and $e'\in E(G)$ with $f(e')\neq 0$.  By Theorem \ref{max_frac_matching} we have that $\sum_{e \in E(G)}f(e) = \mu_f(G) = \frac{1}{2}(|V(G)|-n)$ for
an integer $n \geq 0$. Let $f_0$ be a maximum fractional matching 
with $f_0(e') \not= 0$ and $|\{e: e \in E(G) \text{ and } f_0(e)=0 \}|$ maximal, and let $H=G[\supp(f_0)]$.
We will prove the statement by induction on $n$.

\underline{$n=0$:}
In this case, $f$ and $f_0$ are fractional perfect matchings of $G$, and our proof of the statements closely 
follows the line of the proof of Theorem \ref{max_frac_matching} given in \cite{Scheinerman_Buch}. 

If $H$ contains an edge $e_0=vw$ with $d_H(v)=1$, then $f_0(e_0)=1$ and
$e_0$ is the edge of a $K_{1,1}$-component of $H$. Hence, $f_0(e)=0$ for all $e \in (E(v) \cup E(w))\setminus \{e_0\}$.
In particular, $e' \not \in (E(v) \cup E(w)) \setminus \{e_0\}$. 

\begin{claim} \label{even_circuit}
$H$ does not contain an even circuit.
\end{claim} 
Suppose to the contrary that it contains an even circuit $C$. Let $E(C) = \{e_1, \dots,e_{2k}\}$ and if $e' \in E(C)$, then let $e'=e_1$.
Let $m = \min\{f_0(e_{2i}) : 1 \leq i \leq k\}$. Define $g : E(G) \rightarrow \{-1,0,1\}$, with $g(e)=0$ if $e \in E(G) - E(C)$ and
for $i,j \in \{1,\dots,k\}$ let $g(e_{2i-1})=1$ and $g(e_{2j})=-1$ . Then $f_1=f_0+mg$ is a maximum fractional matching
with $f_1(e') \not=0$ and which assigns 0 to at least one more edge than $f_0$, a contradiction.

\begin{claim} \label{odd_circuit}
If $H$ contains an odd circuit $C_1$, then $C_1$ is a circuit component of $H$. 
\end{claim} 
Suppose that $C_1$ contains a vertex $v$ with $d_H(v) > 2$. Let $P$ be a path which starts in $v$ with an edge which is not an edge of $C_1$.   
This path cannot return to $C_1$, since then $H$ would contain an even circuit. It can also not have an end vertex $x$ of degree 1, since then
$f_0(e)=1$ for the edge which is incident to $x$ in $H$. Hence,
it ends at a vertex $w$ with $N(w) \subseteq V(P)$. Thus, $H$ contains a graph $B$ which consists of 
two odd circuits $C_1$ and $C_2$ which are connected by a path (possibly of length 0). Let $g : E(H) \rightarrow \{-1, -\frac{1}{2}, 0, \frac{1}{2},1\}$ be a function with
$g(e) = 0$ if $e \not \in E(B)$ and $\pm1$ alternately on the path which connects the two odd circuits of $B$ and $\pm \frac{1}{2}$ alternately
around the circuits such that $\sum_{e \in E(v)}g(e)=0$ for each $v \in V(B)$. If $e' \in E(B)$, then choose $g$ such that $g(e') > 0$. 
Let $m$ be the smallest number such that there is an edge $e \in E(B)$ with $f_1(e)=(f_0+mg)(e)=0$. Then $f_1$ is fractional perfect matching 
of $G$ which assigns the value 0 to more edges that $f_0$. Furthermore, the value 0 can only achieved on an edge $e$
with $g(e) < 0$. Hence, $f_1(e') \not=0$ and we obtain a contradiction to the definition of $f_0$. Thus, the claim is proved.

Hence, the components of $H$ are odd circuits or $K_{1,1}$'s. The function $f' : E(G) \rightarrow \{0, \frac{1}{2}, 1\}$ with
$f'(e)=\frac{1}{2}$, if $e$ is an edge of a circuit component of $H$, $f'(e)=1$, if $e$ is an edge of a $K_{1,1}$ component of $H$
and $f'(e)=0$, if $e \not \in E(H)$ is the desired fractional perfect matching of $G$ with $f'(e') \not= 0$. 

\underline{$n \geq 1$:} For $v \in V(G)$ let $\delta_f(v) = 1 - \sum_{e \in E(v)}f(e)$. 
Let $\{v_1, \dots,v_t\}$ be the set of vertices $v$ of $G$ with $\delta_f(v) > 0$. Add a vertex $x$ 
and edges $xv_i$ for $i \in \{1, \dots,t\}$ to $G$ to obtain a new graph $G_x$. Note that $|V(G_x)|=|V(G)|+1$. 

Extend $f$ to a function $h : E(G_x) \rightarrow [0,1]$ with $h(e)=f(e)$ if $e \in E(G)$ and for the edges $xv_1, \dots, xv_t$,
choose $h(xv_i)$ appropriately such that $0 \leq h(xv_i) \leq \delta_f(v_i)$ and $\sum_{i=1}^t h(xv_i) = 1$. 
The function $h$ is a fractional matching on $G_x$. 
It holds that $\sum_{e \in E(G_x)} h(e) = 1 + \sum_{e \in E(G)} f(e) = 1 + \frac{1}{2}(|V(G)|-n) = \frac{1}{2}(|V(G_x)| - (n-1))$.

\begin{claim}
$h$ is a maximum fractional matching of $G_x$.
\end{claim}
If $n=1$, then $h$ is a fractional perfect matching of $G_x$ and therefore, it is maximum. 

For $n \geq 2$ we suppose to the contrary that the graph $G_x$ has a fractional matching $h_0$ with  
$\sum_{e \in E(G_x)}h_0(e) = \frac{1}{2}(|V(G_x)| - m)$ and $m < n-1$. It follows that
$\sum_{e \in E(G)}h_0(e) \geq (\sum_{e \in E(G_x)}h_0(e)) - 1 = \frac{1}{2}(|V(G_x)| - m) - 1
> \frac{1}{2}(|V(G)| - n) = \mu_f(G)$, a contradiction and the claim is proved.

By definition, $h(e') = f(e') \not=0$ and therefore,
$h$ is a maximum fractional matching on $G_x$ with $h(e') \not= 0$ and  $\sum_{e \in E(G_x)} h(e) = \frac{1}{2}(|V(G_x)| - (n-1))$.

By induction hypothesis, there is a maximum fractional matching $h'$ of $G_x$ with 
$h'(e) \in \{0, \frac{1}{2}, 1\}$ for all $e \in E(G)$ and $h'(e') \not= 0$.
Since $\sum_{e \in E(G_x)}h'(e) = 1 + \sum_{e \in E(G)}f(e)$ it follows that $\sum_{e \in E(x)}h'(e) = \sum_{i=1}^t h'(xv_i) = 1$.
Suppose to the contrary that $x$ is a vertex of a circuit component $C$ of $G_x[\supp(h')]$. Since $C$ is an odd circuit, $C-x$ has a perfect matching. 
Thus, $\mu_f(G) > \sum_{e \in E(G)} f(e)$, a contradiction. Hence, $x$ is a vertex of a $K_{1,1}$-component of $G_x[\supp(h')]$, and
$f' : E(G) \rightarrow \{0,\frac{1}{2},1\}$ with $f'(e) = h'(e)$ for all $e \in E(G)$ is the desired maximum fractional matching of $G$.  
\end{proof} 
 \setcounter{case}{0}
\setcounter{claim}{0}

A star-cycle factor $F$ is minimal if $\sum_{i=1}^\infty (i-1)t_i^F = l(G)$.

\begin{coro} \label{e_in_max_cs-factor}
Let $G$ be a graph and $e' \in E(G)$. There is a maximum fractional matching $f$ of $G$ with $f(e') \not= 0$ if and only if
$e'$ is an edge of a minimal star-cycle factor of $G$. 
\end{coro}
\begin{proof}
($\Rightarrow$) Let $\mu_f(G) = \frac{1}{2}(|V(G)| - n)$ for an integer $n \geq 0$. By Theorem \ref{fm_edge} there is a  maximum fractional matching 
$f'$ with $f'(e) \in \{0, \frac{1}{2}, 1\}$ for all $e \in E(G)$ and $f'(e') \not = 0$. Hence, $e'$ is an edge of
a circuit or a $K_{1,1}$-component of $G[\supp(f')]$. Furthermore, there are precisely $n$ vertices $v_1, \dots,v_n$ with
$\sum_{e \in E(v_i)}f'(e)=0$. Let $x \in N(v_i)$. Then $\sum_{e \in E(x)}f'(e)=1$. If $x$
is a vertex of a circuit component $C$ of $G[\supp(f')]$, then, since $C$ is of odd order, we easily deduce a 
contradiction to the maximality of $f'$. Hence, $x \in N(v_i)$ is a vertex of a $K_{1,1}$-component of $G[\supp(f')]$.
Furthermore, at most one end vertex of a $K_{1,1}$-component can be in $\bigcup_{i=1}^n N(v_i)$, since for otherwise we again can deduce a contradiction to the maximality of $f'$. 
Extending $G[\supp(f')]$ by connecting each $v_i$ to one of its neighbors yields the 
desired $\{K_{1,1}, \dots, K_{1,t}, C_m : m \geq 3\}$-factor of $G$. 
The other direction of the statement is trivial.  
\end{proof}

If $\iso(G-S)\le \lambda \vert S\vert$, with $\lambda$ minimal, then the star-cycle factor $F$ 
in Corollary \ref{e_in_max_cs-factor} is not necessarily 
a $\{K_{1,1},\dots K_{1,t}, C_m\colon m\ge 3\}$-factor with $t\le \lambda$.
Recall that $\min(G, K_{1,2}) = \min\{t_2^F : F \text{ is a } \{K_{1,1}, K_{1,2}, C_m : m \ge 3\}\text{-factor of } G \}$.
The following theorem will be used in Section \ref{Edge-colorings}.

\begin{theo}\label{kez_condition}
If a graph $G$ has a $\{K_{1,1}, K_{1,2}, C_m : m \ge 3\}$-factor,
then $\mu_f(G) = \frac{1}{2}(|V(G)|-\min(G, K_{1,2}))$. 
\end{theo}
\begin{proof}
The result follows directly from Theorem \ref{star_cycle_factor_neu} and Corollary \ref{sum_up}.
\end{proof}

Theorem \ref{Tutte_53} is the special case $m=n$ of the following corollary. 
\begin{coro}\label{bounds} 
 Let $G$ be a graph and let $n,m$ be integers with $0 < n \leq m \leq 2n$. If
$\iso(G-S)\le \frac{m}{n}\vert S\vert$ for all subsets $S\subseteq V(G)$,
 then  
 \begin{enumerate}
 \item[$(i)$] $\min(G, K_{1,2})\le \frac{m-n}{m+n} \vert V(G)\vert$,
 \item[$(ii)$] $\alpha(G)\le \frac{m}{m+n} \vert V(G)\vert$.
 \end{enumerate}
\end{coro}

\begin{proof}

$(i)$ Since $1 \le \frac{m}{n} \le 2$ it follows with Theorem \ref{theorem_Akiyama_Era} that  
$G$ has a $\{K_{1,1}, K_{1,2}, C_m\colon m\ge 3\}$-factor.
Furthermore, for all $S \subseteq V(G)$:
 \begin{align*}
 \iso(G-S)&\le  \frac{m}{n}\vert S\vert=  \frac{2m}{2n}\vert S\vert\\
  \Leftrightarrow \frac{2n}{m+n} \iso(G-S) &\le \frac{2m}{m+n} \vert S\vert\\
  \Leftrightarrow   \iso(G-S) -\frac{m-n}{m+n} \iso(G-S)&\le \vert S\vert+\frac{m-n}{m+n} \vert S\vert\\
  \Leftrightarrow  \iso(G-S) &\le \vert S\vert +\frac{m-n}{m+n}\left(\iso(G-S)+ \vert S\vert\right).
 \end{align*}
Since $\iso(G-S)+ \vert S\vert\le \vert V(G)\vert$ for  all $S \subseteq V(G)$ it follows that
\begin{align*}
 \iso(G-S)\le \vert S\vert +\frac{m-n}{m+n}\vert V(G)\vert.
\end{align*}
Now, the result follows with Theorem \ref{kez_condition} and Corollary \ref{sum_up}.

$(ii)$ By $(i)$, $G$ has as a  $\{K_{1,1}, K_{1,2}, C_m\colon m\ge 3\}$-factor $F$ with 
$\min(G, K_{1,2})\le \frac{m-n}{m+n} |V(G)|$. Then, for all $S \subseteq V(G)$ we have
\begin{align*}
 \iso(G-S)&\le \iso(F-S) \le 2 \frac{ m-n}{m+n} \vert V(G)\vert + \frac{1}{2} \left(\vert V(G)\vert-3\frac{ m-n}{m+n} \vert V(G)\vert\right) \\
 &\le \frac{m-n}{2(m+n)} \vert V(G)\vert + \frac{1}{2}\vert V(G)\vert =\frac{m}{m+n}\vert V(G)\vert
\end{align*}
\end{proof}

In the following we will apply Lov\'{a}sz' $(g,f)$-factor Theorem, which is on multigraphs. 

\begin{theo}[\cite{Lovasz1970}] \label{gf_factor_theorem}
Let $G$ be a multigraph and let $g,f:V(G)\to \Z$ be functions such that $g(v)\le f(v)$ for all $v\in V(G)$. Then $G$ has a $(g,f)$-factor if and only if for all disjoint subsets $S$ and $T$ of $V(G)$,
\begin{align*}
 \gamma(S,T)&=\sum_{v\in S} f(v) + \sum_{v\in T} \left(d_G(v) -g(v)\right)-|E_G(S,T)|-q^{\star}(S,T)\\
 &=\sum_{v\in S} f(v) + \sum_{v\in T} \left(d_{G-S}(v) -g(v)\right)-q^{\star}(S,T) \ge 0, \notag
\end{align*}
where $q^{\star}(S,T)$ denotes the number of components $C$ of $G-(S\cup T)$ such that $g(v)=f(v)$ for all $v\in V(C)$ and 
\begin{align*}
 \sum_{v\in V(C)} f(v)+|E_G(C,T)|\equiv 1 \mod 2.
\end{align*}
\end{theo}
Notice that  $q^{\star}(S,T)=0$ for all disjoint subsets $S$ and $T$ of $V(G)$, if $g(v)< f(v)$ for all $v\in V(G)$. \\

The following theorem extends a result of Berge and Las Vergnas (Theorem 7 in \cite{Berge_star_cycle_factor1978}) 
from $[1,2]$-factors to $\{K_{1,1}, K_{1,2}, C_m\colon m\ge 3\}$-factors of a graph.

\begin{theo}\label{Lemma_condition_not_factor}
 Let $G$ be a graph that has a $\{K_{1,1}, K_{1,2}, C_m\colon m\ge 3\}$-factor. 
 For $e \in E(G)$, say $e=uv$, there is no $\{K_{1,1}, K_{1,2}, C_m\colon m\ge 3\}$-factor which contains $e$ 
if and only if there is a subset $S$ of $V(G)$ that satisfies
\begin{enumerate}
  \item [$(i)$] $u,v\in S$
  \item [$(ii)$] $2\vert S\vert -2\le \iso(G-S) \le 2\vert S\vert$.
 \end{enumerate}
Furthermore, the inequalities of $(ii)$ are tight. 
\end{theo}

\begin{proof}
The condition $\iso(G-S) \le 2\vert S\vert$ in (ii) is satisfied, since $G$ has a  
$\{K_{1,1}, K_{1,2}, C_m\colon m\ge 3\}$-factor. Therefore, it remains to prove that $2\vert S\vert -2\le \iso(G-S)$.
We first consider the graph $G'$ which is obtained from $G$ by contracting $e$,
that is $V(G')= (V(G)\setminus \{u,v\}) \cup \{w\} $ 
and $E(G')$ is obtained from $E(G[V(G)\setminus\{u,v\}])\cup\{xw\colon\ xu\in E(G) \text{ or } xv\in E(G)\}$. Notice that $G'$ is not necessarily a simple graph.  
Let $S$ be a subset of $V(G)$ and $S'$ a subset of $V(G')$. Then we call the sets $S$ and $S'$ 
corresponding sets, if $S\setminus\{u,v\}=S'\setminus\{w\}$, and $u,v\in S$ if and only if $w\in S'$.
\begin{claim}\label{claim1_gf_factor}
 $G$ has a $\{K_{1,1}, K_{1,2}, C_m\colon m\ge 3\}$-factor $F$ with $e\in F$ if and only if $G'$ has 
a $(g',f')$-factor with $g'(x)=1$, $f'(x)=2$ for all $x\in  V(G')\setminus \{ w \}$, $g'(w)=0$ and $f'(w)=1$.
\end{claim}
If $G$ has a $\{K_{1,1}, K_{1,2}, C_m\colon m\ge 3\}$-factor $F$ with $e\in F$ and $e$ is contained in a $C_m$-component, then  decompose this component into $K_{1,1}$ and $K_{1,2}$-components. So $e$ is either contained in a $K_{1,1}$-component or in a $K_{1,2}$-component. Contract $e$, and the remaining edges of $F$ in $G'$ obviously form a $(g',f')$-factor of $G'$.

If $G'$ has a $(g',f')$-factor $F'$, then $g'(w) \in \{0,1\}$. 
If $g'(w)=0$, then let $F=F'\cup\{u,v\}-w$. Otherwise, assume $w'\in N_{F'}(w)$ and 
$vw' \in E(G)$ and let $F=F'\setminus\{ww'\}\cup \{uv, vw'\}$. Then $F$ is a $[1,2]$-factor of $G$ 
and in any case, $e$ is an end edge of a path. 

If we decompose all paths of length at least three into paths of length 
one or two, then we get a $\{K_{1,1}, K_{1,2}, C_m\colon m\ge 3\}$-factor $F''$ of $G$ with $e\in F''$, and the claim is proved.

\glqq $\Leftarrow$ \grqq: 
 Let $S$ be a set of $V(G)$ with $u,v\in S$ and  $2\vert S\vert -2\le \iso(G-S)$. 
 Let $S'$ be the corresponding set of $S$. Since $u,v\in S$, we have $w\in S'$. Further $\vert S\vert = \vert S'\vert +1$, $\iso(G-S)=\iso(G'-S')$ and $2\vert S'\vert \le \iso(G'-S')$.\\
Let $T':=\Iso(G'-S')$ and let $f'$, $g'$ be the same as in Claim \ref{claim1_gf_factor}. Then it follows
 
\begin{align*}
  \sum_{x\in S'} f'(x) +\sum_{x\in T'}( d_{G'-S'} (x) - g'(x))= 2 \vert S'\vert - 1 - \vert T'\vert \le -1.
 \end{align*}
By Theorem \ref{gf_factor_theorem}, $G'$ has no $(g',f')$-factor and by Claim \ref{claim1_gf_factor} $G$ has no $\{K_{1,1}, K_{1,2}, C_m\colon m\ge 3\}$-factor that contains $e$.

\glqq $\Rightarrow$ \grqq:
Let $e$ be an edge of $E(G)$, say $e=uv$, that is not contained in any $\{K_{1,1}, K_{1,2}, C_m\colon m\ge 3\}$-factor of $G$. \\
 Since $G$ has a $\{K_{1,1}, K_{1,2}, C_m\colon m\ge 3\}$-factor, $G$ also has a $(g,f)$-factor with $g(x)=1$ and $f(x)=2$ for all $x\in V(G)$ and by Theorem \ref{gf_factor_theorem} for all disjoint subsets $X$ and $Y$ of $V(G)$ we have
 \begin{align}\label{characterization_gf_factor}
  \gamma(X,Y)=\sum_{x\in X} f(x)+\sum_{y\in Y} (d_{G-X}(y)-g(y))\ge 0.
 \end{align}
Since $e$ is not contained in any $\{K_{1,1}, K_{1,2}, C_m\colon m\ge 3\}$-factor of $G$, by Claim \ref{claim1_gf_factor} and Theorem \ref{gf_factor_theorem}, there are two disjoint subsets $X'$ and $Y'$ of $V(G')$ with $\gamma(X',Y')<0$ (with respect to $g'$ and $f'$). Let $S'$ and $T'$ be two subsets of $V(G')$ satisfying $\gamma(S',T')<0$.

Case 1: $w\notin S'\cup T'$. 
We have
\begin{align*}
 \gamma(S',T')&=  \sum_{x\in S'} f'(x)+\sum_{x\in T'} (d_{G'-S'}(x)-g'(x))=  \sum_{x\in S} f(x)+\sum_{x\in T} (d_{G-S}(x)-g(x))=\gamma(S, T).
\end{align*}
This is a contradiction, since by inequality (\ref{characterization_gf_factor}) it follows that $\gamma(S', T')\ge 0$.

Case 2: $w\in T'$. We have
\begin{align*}
 \gamma(S', T')&=\sum_{x\in S'} f'(x)+\sum_{x\in T'}(d_{G'-S'}(x)-g'(x))\\
 &= \sum_{x\in S'} f'(x)+\sum_{x\in T'\setminus w}(d_{G'-S'}(x)-g'(x))+d_{G'-S'}(w)-g'(w)\\
 &=\underset{\ge0}{\underbrace{\sum_{x\in S} f(x)+\sum_{x\in T\setminus \{u,v\}}(d_{G-S}(x)-g(x))}}+d_{G'-S'}(w)-0 \ge 0,
\end{align*}
again a contradiction.

Case 3: $w\in S'$. We have
\begin{align*}
 \gamma(S', T')&=\sum_{x\in S'} f'(x)+\sum_{x\in T'}(d_{G'-S'}(x)-g'(x))\\
 & = 2\vert S'\vert -1 - \vert T'\vert + \sum_{x\in T'} d_{G'-S'} (x)  <0 \notag
 \end{align*}
and, since $\gamma(S', T')$ is a natural number, it follows, that
\begin{align}\label{inequality_gf_3}
  \sum_{x\in T'} d_{G'-S'}(x) \le \vert T'\vert -2 \vert S'\vert.
\end{align}
Since  $\sum_{x\in T'} d_{G'-S'}(x)\ge 0$, we have $\vert T'\vert \ge 2 \vert S'\vert$.

Suppose $\iso(G'-S') < 2\vert S'\vert $. It follows that  $\sum_{x\in T'} d_{G'-S'}(x) \ge \vert T'\vert -2\vert S'\vert +1$, a contradiction by the right side of inequality (\ref{inequality_gf_3}). Therefore, $\iso(G'-S')\ge 2\vert S'\vert $.\\
We have $\vert S\vert = \vert S' \vert +1$ and $\iso(G-S)=\iso(G'-S')$.
Therefore, there is a subset $S$ of $V(G)$ with $u,v\in S$ and $2\vert S\vert -2\le \iso(G-S)$, if there is no $\{K_{1,1}, K_{1,2}, C_m\colon m\ge 3\}$-factor that contains $e$.

 \setcounter{case}{0}
\setcounter{claim}{0}

We give some examples to show that the inequalities of $(ii)$ are tight. 
\begin{itemize}
 \item For the given graph there is no $\{K_{1,1}, K_{1,2}, C_m\colon m\ge 3\}$-factor that contains the edge $e=uv$ and for $S=\{u,v\}$ we have  $\iso(G-S)=2\vert S\vert $  

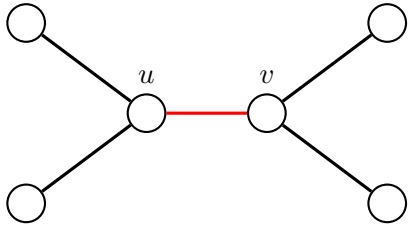
\begin{tikzpicture}[scale=0.8, 
every edge/.style = {draw=black,very thick},
 vrtx/.style args = {#1/#2}{%
      circle, draw, thick, fill=white,
      minimum size=5mm, label=#1:#2}
                    ]
\node(E) [vrtx=above/] at (6,0.5) {};
\node(F) [vrtx=above/] at (6,3.5) {};
\node(D) [vrtx=above/$v$] at (4,2) {};
\node(C) [vrtx=above/$u$] at (2,2) {};
\node(A) [vrtx=above/] at (0,0.5) {};
\node(B) [vrtx=above/] at (0,3.5) {};

\path   (A) edge (C)
        (B) edge (C)
        (C) edge[red] (D)
        (D) edge (E)
        (D) edge (F);

\end{tikzpicture}
 \item For the given graph there is no $\{K_{1,1}, K_{1,2}, C_m\colon m\ge 3\}$-factor that contains the edge $e=uv$ and for $S=\{u,v\}$ we have  $\iso(G-S)=2\vert S\vert -1$  

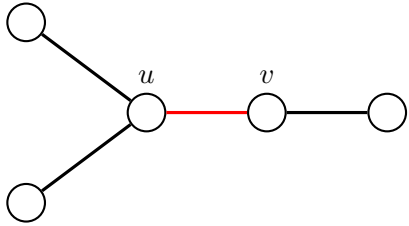
\begin{tikzpicture}[scale=0.8, 
every edge/.style = {draw=black,very thick},
 vrtx/.style args = {#1/#2}{%
      circle, draw, thick, fill=white,
      minimum size=5mm, label=#1:#2}
                    ]
\node(E) [vrtx=above/] at (6,2) {};
\node(D) [vrtx=above/$v$] at (4,2) {};
\node(C) [vrtx=above/$u$] at (2,2) {};
\node(A) [vrtx=above/] at (0, 0.5) {};
\node(B) [vrtx=above/] at (0,3.5) {};

\path   (A) edge (C)
        (B) edge (C)
        (C) edge[red] (D)
        (D) edge (E);

\end{tikzpicture}
 \item For the given graph there is no $\{K_{1,1}, K_{1,2}, C_m\colon m\ge 3\}$-factor that contains the edge $e=uv$ and for $S=\{u,v, v_1, v_2\}$ we have $\vert S\vert=4$, $\iso(G-S)=6$. Thus, $\iso(G-S)=2\vert S\vert-2 $  

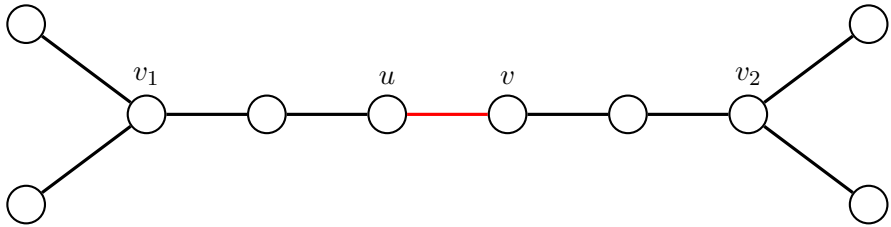
\begin{tikzpicture}[scale=0.8, 
every edge/.style = {draw=black,very thick},
 vrtx/.style args = {#1/#2}{%
      circle, draw, thick, fill=white,
      minimum size=5mm, label=#1:#2}
                    ]
                    ]
\node(J) [vrtx=above/] at (14,0.5) {};
\node(I) [vrtx=above/] at (14,3.5) {};
\node(H) [vrtx=above/$v_2$] at (12,2) {};
\node(G) [vrtx=above/] at (10,2) {};

\node(F) [vrtx=above/$v$] at (8,2) {};
\node(E) [vrtx=above/$u$] at (6,2) {};
\node(D) [vrtx=above/] at (4,2) {};
\node(C) [vrtx=above/$v_1$] at (2,2) {};

\node(A) [vrtx=above/] at (0,0.5) {};
\node(B) [vrtx=above/] at (0,3.5) {};

\path   (A) edge (C)
        (B) edge (C)
        (C) edge (D)
        (D) edge (E)
        (E) edge[red] (F)
        (F) edge (G)
        (G) edge (H)
        (H) edge (I)
        (H) edge (J);   

\end{tikzpicture}
\end{itemize}
\end{proof}

\begin{coro}
Let $G$ be a graph that has a $\{K_{1,1}, K_{1,2}, C_m\colon m\ge 3\}$-factor and $e \in E(G)$.
If $e$ is not contained in any $\{K_{1,1}, K_{1,2}, C_m\colon m\ge 3\}$-factor, then $f(e)=0$ for
every maximum fractional matching $f$ of $G$. 
\end{coro}

\begin{proof}
By Theorem \ref{theorem_Akiyama_Era} we have $\iso(G-S)  \le 2 \vert S \vert$ for all $S \subseteq V(G)$.
Hence, $G$ has a maximum $\{K_{1,1}, K_{1,2}, C_m\colon m\ge 3\}$-factor by Theorem \ref{star_cycle_factor_neu}.
In particular, $e$ is not an edge of any maximum $\{K_{1,1}, K_{1,2}, C_m\colon m\ge 3\}$-factor and it
follows with Corollary \ref{e_in_max_cs-factor} that $f(e) = 0$ for every maximum fractional matching of $G$.   
\end{proof}


\section{Component factors of edge-chromatic critical graphs} \label{Edge-colorings}

Woodall \cite{W2011_1} proved that $\alpha(G) \leq \frac{3}{5}|V(G)|$ for a critical graph $G$. 
Using his proof approach we generalize some of his results to deduce that every critical graph has 
a $[1,2]$-factor. The components of $[1,2]$-factors are paths and circuits. A path with an odd (even) number of vertices is called an odd (even) path. The length of a path is the number of edges appearing in it. Clearly, every $[1,2]$-factor can be 
decomposed into a $\{K_{1,1}, K_{1,2}, C_m\colon m\ge 3\}$-factor. We will use this fact to prove an 
upper bound for $\min(G, K_{1,2})$, for critical graphs. As a reminder, $\min(G, K_{1,2})=\min\{t_2^F : F \text{ is a } \{K_{1,1}, K_{1,2}, C_m : m \ge 3\}\text{-factor of } G \}$, where $t_2^F$ is the number of $K_{1,2}$ components of a $\{K_{1,1}, K_{1,2}, C_m\colon m\ge 3\}$-factor $F$. Every odd path of length $n$ can be decomposed into $\frac{n}{2} -1$ $K_{1,1}$-components and one $K_{1,2}$-component and every even path of length  $m$ can be decomposed into $\lceil \frac{m}{2} \rceil$ $K_{1,1}$-components. Therefore, the minimal number of odd paths of a $[1,2]$-factor equals $\min(G, K_{1,2})$.

\begin{lem} [Vizing's Adjacency Lemma \cite{V1965}] \label{VAL}
Let $G$ be a critical graph. If $e=xy \in E(G)$, then at least $\Delta(G)-d_G(y)+1$ vertices in 
$N(x)\setminus \{y\}$ have degree $\Delta(G)$.
\end{lem}

Let $G$ be a critical graph. If $vw$ is an edge of $G$, then we denote by $\sigma(v,w)$ the number of vertices in $N(w)\setminus\{v\}$ that have degree 
at least $2\Delta(G)-d_G(v)-d_G(w)+2$. We have $2\Delta(G)-d_G(v)-d_G(w)+2\le \Delta(G)$, since in a critical 
graph $G$, $d_G(v)+d_G(w)\ge \Delta(G)+2$. Further, we have
 \begin{align}\label{inequality_s}
  \sigma(v,w)\ge \Delta(G)-d_G(v)+1,
 \end{align}
since by Lemma \ref{VAL}, $w$ has at least $\Delta(G)-d_G(v)+1$ neighbors different from $v$ with degree $\Delta(G)$.

\begin{lem}[\cite{W2007_II}]\label{p_lemma}
Let $G$ be a critical graph and $v\in V(G)$ and let
\begin{align} \label{p}
 p_{min}:=\min_{w\in N(v)} \sigma(v,w)-\Delta(G)+d_G(v)-1 ~~\text{and}~~ p:=\min\left\{  p_{min}, \left\lfloor \frac{1}{2}d_G(v)\right\rfloor-1\right\}.
\end{align}
Then $v$ has at least $d_G(v)-p-1$ neighbors $w$ for which $\sigma(v,w)\ge \Delta(G)-p-1$.
\end{lem}

\begin{theo}\label{LemmaEstimation}
 Let $G$ be a critical graph and let $S$ be an arbitrary subset of $V(G)$. Then 
 \begin{align*}
\iso(G-S) < \left(\frac{3}{2}-\frac{1}{\Delta(G)} \right) \vert S\vert.
 \end{align*}
\end{theo}

\begin{proof}
Let $G$ be a critical graph, $S$ be an arbitrary subset of $V(G)$ and $T=\Iso(G-S)$. Further let $T^-=\{ t\in T \colon\  2\le d_G(t) <\frac{1}{2}\Delta(G)\}$, $T^+=\{ t\in T \colon\  \frac{1}{2}\Delta(G)\le d_G(t) <\Delta(G)\}$, and $T^{++}=\{ t\in T \colon\   d_G(t) =\Delta(G)\}$. In a critical graph there are no vertices of degree less than 2, so $T=T^-\cup T^+ \cup T^{++}$.\\
We define two functions $f_i: T\to \R$ with $f_i(t)=g_i(d_G(t))$ for all vertices $t\in T$ and $i\in\{1,2\}$, where $g_i:\N \to \R$ and
\begin{align*}
 g_1(k):=\frac{2 (\Delta(G)-k)}{k} ~~\text{and}~~ g_2(k):= \frac{\Delta(G)-2}{k-1}. 
\end{align*}
The functions $g_1$ and $g_2$ are both decreasing functions of $k$.
\begin{claim} \label{FirstInequality}
 For all $t\in T^+$, $f_1(t)\le f_2(t)$.
\end{claim}
\begin{proof}Let $t$ be a vertex of $T^+$ and $k:=d_G(t)$. Then
\begin{align*}
 f_2(t)-f_1(t)&=g_2(k)-g_1(k)=\frac{(\Delta(G)-2)k-2(\Delta(G)-k)(k-1)}{k (k-1)}\\
 &=\frac{(2k-\Delta(G))(k-2)}{k (k-1)}\ge 0
\end{align*}
since $k\ge \frac{1}{2}\Delta(G)$ and $k\ge 2$. Thus, the claim is proved. 
\end{proof}
We now define three charge functions $M_i$, $i\in\{0,1,2\}$ on $V(G)$ as follows: $M_i:V(G)\to \N$ with
\begin{align*}
 M_0(t)&=0, 		 &M_1(t)&=2 d_G(t),~~ 	& M_2(t)&=2\Delta(G) 	& \text{if}~ t\in T,\\
  M_0(s)&=3\Delta(G)-2, 	 &M_1(s)&=\Delta(G)-2,~~ 	& M_2(s)&=0	& \text{if}~ s\in S,\\
  M_0(v)&=0, 		&M_1(v)&=0,~~ 		& M_2(v)&=0, 	& \text{if}~v\in V(G)-(S\cup T).
\end{align*}
We will prove that the functions $M_1$ and $M_2$ satisfy
\begin{enumerate}
 \item[(i)] $ \sum_{v\in V(G)} M_1(v) < (3\Delta(G)-2) \vert S\vert$,\label{condition_i}
 \item[(ii)]$\sum_{v\in V(G)} M_2(v)\le\sum_{v\in V(G)} M_1(v)$.\label{condition_ii}
\end{enumerate}

This will imply 
\begin{align*}
 2\Delta(G)\vert T\vert = \sum_{v\in V(G)} M_2(v) \le \sum_{v\in V(G)} M_1(v) < (3\Delta(G)-2) \vert S\vert
\end{align*}
and therefore,
\begin{align*}
\iso(G-S)= \vert T \vert < \left(\frac{3}{2}-\frac{1}{\Delta(G)} \right) \vert S\vert,
\end{align*}
which is the required result.

\begin{proof}[Proof of (i)]
 Starting with the distribution $M_0$, let each vertex in $T$ receive charge 2 from each of its neighbors in $S$. Let the resulting charge distribution be called $M_0^\star$. We have $M_0^\star(t)= 2 d_G(t)$ for all $t\in T$ and for all $s\in S$, $M_0^\star(s)= 3\Delta(G)-2 -2 \vert N(s)\cap T\vert \ge \Delta(G)-2$. So $M_0^\star(v)\ge M_1(v)$ for all $v\in V(G)$, with strict inequality if $s$ is a vertex of $S$ with fewer than $\Delta(G)$ neighbors in $T$. There exists such a vertex $s$, since either $s$ has a neighbor in $V(G)\setminus(S\cup T)$ or $S\cup T=V(G)$ and $S$ is not an independent set, since a critical graph cannot be bipartite. Thus, $\sum_{v\in V(G)} M_1(v)< \sum_{v\in V(G)} M_0^\star(v) =\sum_{v\in V(G)} M_0(v)= (3\Delta(G)-2)\vert S\vert$. This proves (i).
 \end{proof}
\begin{proof}[Proof of (ii)]
Starting with the distribution $M_1$, we will redistribute charge according to the following discharging rule:
\begin{itemize}[leftmargin=8pt]
 \item[-] Step 1: Each vertex $s\in S$ gives charge $f_1(t)$ to each vertex $t\in N(s)\cap T^+$.
 \item[-] Step 2: Each vertex $s\in S$ distributes its remaining charge equally among all vertices (if any) in $N(s)\cap T^-$.
\end{itemize}
The resulting charge distribution we denote by $M^\star_1$. 

\begin{claim}\label{claim2}
 $M^\star_1(s)\ge 0=M_2(s)$ for all $s\in S$. 
\end{claim}
\begin{proof}
We compare the above discharging rule, the actual discharging rule, with the equitable discharging rule in which each vertex $s\in S$ distributes its charge of $M_1(s)=\Delta(G)-2$ equally among all its neighbors (if any) in $T^-\cup T^+$. 
Let $s\in S$ and let $\delta$ be the minimum degree of the neighbors of $s$. By Lemma \ref{VAL} the vertex $s$ has at least $\Delta(G)-\delta+1$ neighbors of degree $\Delta(G)$, and hence, at most $\delta-1$ neighbors in $T^-\cup T^+$. Thus, under the equitable discharging rule, each vertex $t\in N(s)\cap T^+ $ receives from $s$ at least
\begin{align*}
 \frac{\Delta(G)-2}{\delta-1}\ge \frac{\Delta(G)-2}{d_G(t)-1}=f_2(t)\ge f_1(t),
\end{align*}
by Claim \ref{FirstInequality}. Hence, every vertex of $N(s)\cap T^+$ receives no more charge from $s$ in Step 1  of the actual discharging rule than it would receive under the equitable discharging rule. Thus,  $M^\star_1(s)\ge 0=M_2(s)$ for all $s\in S$. 
\end{proof}
It remains to show that $M^\star_1(t)\ge 2\Delta(G)=M_2(t)$ for all $t\in T$.
For all $t\in T^{++}$, $M^\star_1(t)=2\Delta(G)=M_2(t)$. Further, for all $t\in T^+$, $M^\star_1(t)=2d_G(t)+d_G(t) f_1(t) =2\Delta(G)=M_2(t)$. It remains to consider vertices in $T^-$.\\
We fix a vertex $t\in T^-$ and denote by $k$ the degree of $t$, so $k= d_G(t)$. Further we define a function $h$ with $h: \N\times \N_0 \to \R$ by
\begin{align*}
 h(k,l)&= \frac{1}{k-l-1} \left(\Delta(G)-2 - l g_1(\Delta(G)-k+2)\right)\\
 &= \frac{1}{k-l-1} \left( \Delta(G)-2-l \frac{2 (k-2)}{\Delta(G)-k+2}\right).
\end{align*}
\begin{claim}\label{claim5}
 If $l$ is a nonnegative integer and a vertex $s\in S$ is a neighbor of $t$ such that  $\sigma(t,s)\ge \Delta(G)-k+l+1$, then $s$ gives $t$ at least charge $h(k,l)$  in Step 2.
\end{claim}
\begin{proof}
By definition of $\sigma(t,s)$, vertex $s$ has $\sigma(t,s)$ neighbors with degree at least $2\Delta(G) - k -d_G(s)+2$. Since $d_G(s)\le\Delta(G)$ and  $t\in T^-$ and therefore, $k < \frac{1}{2}\Delta(G)$,
\begin{align*}
 2\Delta(G) - k -d_G(s)+2\ge \Delta(G)-k+2 > \frac{1}{2} \Delta(G).
\end{align*}
By Lemma \ref{VAL}, vertex $s$ has at least $\Delta(G)-k+1$ neighbors with degree $\Delta(G)$. Let $L^{++}$ be a set of $\Delta(G)-k+1$ neighbors of $s$ with degree $\Delta(G)$, and let $L^+$ be a set, disjoint from $L^{++}$, of $l$ neighbors of $s$ with degree at least $\Delta(G)-k+2$, which exists since $\sigma(t,s)\ge \Delta(G)-k+l+1$ by hypothesis. So $L^{++}\subseteq T^{++}\cup S \cup\left( V(G)\setminus (S \cup T)\right)$ and $L^{+}\subseteq T^{++}\cup T^{+}\cup S \cup \left(V(G)\setminus (S \cup T)\right)$. \\
Applying the actual discharging rule, vertex $s$ gives nothing to any vertex in $L^{++}$ and in Step 1  $s$ gives each vertex in $L^{+}$  at most charge $g_1(\Delta(G)-k+2)$, since $g_1$ is a decreasing function and the degree of any vertex in $L^{+}$ is at least $\Delta(G)-k+2$. So the remaining charge of $s$ is at least $\Delta(G)-2-l g_1(\Delta(G)-k+2)$ and there are $d_G(s)-(\Delta(G)-k+l+1)\le k-l-1$ remaining neighbors of $s$. \\
For each vertex  $v\in T^+$
\begin{align*}
\frac{\Delta(G)-2}{k-1}>\frac{\Delta(G)-2}{d_G(v)-1} \ge g_2(d_G(v))\ge g_1(d_G(v)),
\end{align*}
since $d_G(v)>k$ and hence,
\begin{align*}
 &h(k,l)
\ge  \frac{\Delta(G)-2-l\frac{\Delta(G)-2}{k-1}  }{k-l-1}
 =\frac{(\Delta(G)-2) (k-l-1)}{(k-1) (k-l-1)} \ge g_1(d_G(v))=f_1(v).
\end{align*}
Therefore, any vertex in $T^{-}$ gets as least as much of it as any other neighbor of $s$ and therefore, at least $h(k,l)$. Thus, the claim is proved.
\end{proof}

We now prove that vertex $t$ gets at least $2 (\Delta(G)-k)$ charge in Step 2. This implies that  $M^\star_1(t)\ge M_1(t)+2(\Delta(G)-k)=2\Delta(G) = M_2(t)$.\\
We define $p$ as in (\ref{p})  of Lemma \ref{p_lemma}. It follows that $t$ has at least $k-p-1$ neighbors $s\in S$ with $\sigma(t,s)\ge \Delta(G)-p-1$. Let $N^{+}(t)$ be a set of $k-(p+1)$ such neighbors and let $N^{-}(t)=N(t)\setminus N^{+}(t)$. The set $N^{-}(t)$ contains $p+1$ neighbors $s$ of $t$,  each with $\sigma(t,s)\ge \Delta(G)-k+p+1$, by the definition of $p$. Applying Claim \ref{claim5} to the vertices $N^{-}(t)$ with $l=p$ for the vertices in $N^{-}(t)$ and $l=k-p-2$ for the vertices in $N^{+}(t)$, we see that $t$ receives charge of at least $M^+(k,p)$ in Step 2, where
\begin{align*}
 M^+(k,p):= (p+1) h(k,p) + (k-(p+1)) h(k, k-p-2).
\end{align*}
It remains to show that  $M^+(k,p)\ge 2(\Delta(G)-k)$. 
Let $r= p+1$, so that $1\le r\le \frac{1}{2} k$, since $0\le p \le \frac{1}{2}k-1$ by (\ref{inequality_s}) and (\ref{p}). Setting 
\begin{align*}
 b:= \frac{2(k-2)}{\Delta(G)-k+2} \text{ and } a:=\Delta(G)-2+b
\end{align*}
we can write
\begin{align*}
 M^+(k,p)=\frac{r (a-br)}{k-r} + \frac{(k-r) (a-b(k-r))}{r}.
\end{align*}
The derivative of this with respect to $r$ is
\begin{align*}
 \frac{ak-bk^2 + b (k-r)^2}{(k-r)^2 }- \frac{ak-bk^2 + b r^2}{r^2}=\frac{ak-bk^2}{(k-r)^2}-\frac{ak-bk^2}{r^2}.
\end{align*}
This is zero if and only if $r=\frac{1}{2}k$ (unless $ak-bk^2=0$, if $M^+(k,p)$ is independent of $p$); thus, $M^+(k,p)$, regarded as a function of $p$, has only one stationary point (for positive $p$), when $p+1=\frac{1}{2} k$. Substituting this value of $p$ gives
\begin{align*}
 M^+\left(k, \frac{1}{2} k-1\right)=2\left(\Delta(G)-2-\frac{(k-2)^2}{\Delta(G)-k+2}\right) \ge 2 (\Delta(G)-k),
\end{align*}
where the inequality holds because $k <\frac{1}{2}\Delta(G)$ and so
\begin{align*} 
 \frac{(k-2)^2}{\Delta(G)-k+2}\le k-2.
\end{align*}
To complete the proof, we must consider also the other extreme value of $p$, $p=0$, and show that $M^+(k,0)\ge 2(\Delta(G)-k)$, so we have to show that
\begin{align}\label{inequality5}
 \frac{\Delta(G)-2}{k-1}+(k-1)\left( \Delta(G)-2-\frac{2 (k-2)^2}{\Delta(G)-k+2}\right) \ge 2(\Delta(G)-k).
\end{align}
This evidently holds with equality if $k=2$; so we may assume that $k\ge 3$. Since $k<\frac{1}{2} \Delta(G)$, we can write $\Delta(G)=2k+q$, where $q\ge 1$. Ignoring the first term of (\ref{inequality5}), and dividing through by $k-1$ and rearranging, it suffices to show that 
\begin{align}\label{inequality6}
 2k+q-2-\frac{2(k+q)}{k-1}-\frac{2(k-2)^2}{k+q+2}\ge 0.
\end{align}
Since the left side of (\ref{inequality6}) is clearly an increasing function of $q$, it suffices to verify inequality (\ref{inequality6}) for $q=1$, when the left side becomes
\begin{align*}
 2k-1-\frac{2(k+1)}{k-1}-\frac{2k^2-8k+8}{k+3}&=2k-1-2-\frac{4}{k-1}-2k+14-\frac{50}{k+3}\\
 &=11-\frac{4}{k-1}-\frac{50}{k+3},
\end{align*}
which is positive since $k\ge 3$.\\
This completes the proof of (ii) and also of Theorem \ref{LemmaEstimation}.
\end{proof} \end{proof}

\begin{theo} \label{first_thm}
 Let $G$ be a critical graph. Then $G$ has a $\{K_{1,1}, K_{1,2}, C_m\colon m\ge 3\}$-factor with 
$\min(G,K_{1,2}) = |V(G)| - 2\mu_f(G)$. In particular,
$\min(G, K_{1,2}) \le \frac{1}{5} \vert V(G)\vert$ and  $\alpha(G)\le \frac{3}{5}\vert V(G)\vert$ for all $\Delta(G)\ge2$.
\end{theo}

\begin{proof}
Let $G$ be a critical graph and let $S$ be an arbitrary subset of $V(G)$. By Theorem \ref{LemmaEstimation},
$\iso(G-S)  <\left(\frac{3}{2}-\frac{1}{\Delta(G)} \right) \vert S\vert < \frac{3}{2} \vert S\vert$,
and the statement follows by Theorem \ref{kez_condition} and Corollary \ref{bounds}.
\end{proof}

Furthermore, we have: 

\begin{theo} \label{main_new}
Let $G$ be a critical graph. For every edge $e$ there is a $\{K_{1,1}, K_{1,2}, C_m\colon m\ge 3\}$-factor 
$F$ with $e\in E(F)$.
\end{theo}

\begin{proof}
Let $G$ be a critical graph and let $e=vw$. Suppose to the contrary that 
there is no $\{K_{1,1}, K_{1,2}, C_m : m \ge 3\}$-factor that contains $e$. By Theorems \ref{LemmaEstimation}  
and \ref{Lemma_condition_not_factor} there is a subset $S$ of $V(G)$ 
with $u,v\in S$ and $2\vert S\vert -2 \le \iso(G-S) <  \left(\frac{3}{2}-\frac{1}{\Delta(G)}\right) \vert S\vert < \frac{3}{2} \vert S\vert$. Since $u,v\in S$, we have $\vert S \vert \ge 2$ and hence, $\Delta(G)\ge 3$. 

If $\Delta(G)=3$, then 
$2\vert S\vert-2 < \left(\frac{3}{2}-\frac{1}{3}\right) \vert S\vert =  \frac{7}{6}\vert S\vert
 \Leftrightarrow~  \frac{5}{6}\vert S\vert < 2
 \Leftrightarrow ~ \vert S\vert< \frac{12}{5}$.

Since $\vert S\vert$ and $\iso(G-S)$ are integers, $\vert S\vert =2$ and $\iso(G-S)=2$.
Let $v_1, v_2$ be the isolated vertices of $G-S$. Since $G$ is critical and 
$\vert S\vert =2$, $d(v_i)=2$ and $N_G(v_i)=S$, $i\in\{1,2\}$. This is a contradiction, since in a critical graph vertices of degree two have no common neighbor.

If $\Delta(G) \ge 4$, then
 $2\vert S\vert-2 <  \frac{3}{2}\vert S\vert
 \Leftrightarrow~  \frac{1}{2}\vert S\vert < 2
 \Leftrightarrow ~ \vert S\vert < 4$.
 
Since $\vert S\vert$ and $\iso(G-S)$ are integers, there are the following two possibilities.
If $\vert S\vert =2$, then $\iso(G-S)=2$. Again a contradiction.
If $\vert S\vert =3$, then $\iso(G-S)=4$. Since in a critical graph, there are no vertices of degree less than 2, the number of edges in $E_G(S, \Iso(G-S)) \ge 8$. Since the degree of a vertex in $\iso(G-S)$ is at most 3, with Lemma \ref{VAL} a vertex of $S$ has a least $\Delta(G)-2$ vertices of degree $\Delta(G)$ ($\Delta(G) \ge 4$). Therefore, $E_G(S, \Iso(G-S)) \le 6$. A contradiction.
\end{proof}

\section{Fractional matchings on edge-chromatic critical graphs} \label{conjectures}

The study of fractional matchings of critical graphs gives insight into the structure of critical graphs. 
Our studies of component factors of critical graphs use the concept of fractional matchings.  We propose the following conjecture.

\begin{conj}\label{FPMC} 
If $G$ is a critical graph, then $G$ has a fractional perfect matching.
\end{conj}

Conjecture \ref{FPMC} is in between Conjectures \ref{2FC} and \ref{INC}. We have: Conjecture \ref{2FC} implies Conjecture \ref{FPMC},
which implies Conjecture \ref{INC}. Clearly, Conjecture \ref{FPMC} is true for 2-critical graphs. 

For a graph $G$ with $\Delta(G)=k$, the $k$-deficiency of $G$ is $k|V(G)|-2|E(G)|$ and it is denoted by $s(G)$. 
The function $f$ with $f(e) = \frac{1}{k}$ for each $e \in E(G)$ is a fractional matching on $G$. 
Hence, we obtain the following corollary. 

\begin{coro}
If $G$ is a $k$-critical graph, then 
$\mu_f(G) \geq \frac{1}{2}(|V(G)| - \lfloor \frac{s(G)}{k} \rfloor)$, and therefore,
$\min(G, K_{1,2}) \leq \lfloor \frac{s(G)}{k} \rfloor$, and
$\alpha(G) \leq \frac{1}{2}(|V(G)| +  \lfloor \frac{s(G)}{k} \rfloor)$. 
\end{coro}

Let $k \geq 2$ be an integer and $G$ be a graph with $\Delta(G) =k$. Let $v\in V(G)$ with $d_{G}(v)=d$ and  
let $N_G(v) = \{v_{1},v_{2},....,v_{d}\}$. Let $u_{1},....,u_{k}$ be vertices of degree $k-1$
in a complete bipartite graph $K_{k,k-1}$. Graph $G'$ is a Meredith extension \cite{Meredith_1973} of $G$ 
(applied on $v$), if it is obtained from $G-v$ and $K_{k,k-1}$ by adding edges $v_{i}u_{i}$ for each $i \in \{1,...,d\}$. The copy of $K_{k,k-1}$ which replaces $v$ is denoted by $K_{k,k-1}^v$.
In \cite{Stefan_es_1999} it is proved that $G$ is critical if and only if $G'$ is critical. 
Similar to the proofs of the corresponding statements for Conjectures \ref{2FC} and \ref{INC} \cite{BS2017, ES_2018} 
we can apply Meredith extension to prove the following statement. 

\begin{theo}
The following two statements are equivalent for each $k \geq 3$:
\begin{enumerate}
\item Every $k$-critical graph $G$ has a fractional perfect matching.
\item Every $k$-critical graph $G$ with $\delta(G) = k-1$ has a fractional perfect matching. 
\end{enumerate}
\end{theo}

\begin{proof}
Let $G$ be a $k$-critical graph. Apply Meredith extension to all vertices $v$ of $G$ with $d_G(v) < k - 1$. 
The resulting graph $H$ has $\delta(H) = k-1$ and it has a fractional perfect matching
if $G$ has one. 

If $H$ has a fractional perfect matching, then,  
by Theorem \ref{max_frac_matching}, there is one, say $f$, such that $f(e) \in \{0, \frac{1}{2}, 1\}$ 
for all $e \in E(H)$. 
If $u$ is a vertex of $G$ to which Meredith extension was applied on, 
then $|\supp(f) \cap  \partial_H(V(K_{k,k-1}^u))| \in \{1,2\}$. 
In both cases it is easy to see that the contraction of the $K_{k,k-1}$ yields a critical graph 
which has a fractional perfect matching. So eventually $G$ has one.
\end{proof}

Let $G$ be a graph with Gallai-Edmonds decomposition $(D,A,C)$. Liu and Liu \cite{Liu_Liu2002} proved that $\mu_f(G) = \mu(G)$ 
if and only if $D$ is an independent set. In particular,  $\mu_f(G) = \mu(G)$ if $G$ has a 1-factor. Furthermore, 
if $G$ has a 1- or a 2-factor, then $G$ has a fractional perfect matching.  
In \cite{Stefan_es_1999} it is shown that for all $k \geq 3$ there are $k$-critical graphs of even order which
have no 1-factor, and that there are $k$-critical graphs $G$ of odd order and $G-v$ does not have a 1-factor,
where $d_G(v)= \delta(G)$. We propose a conjecture which is unsolved even for critical graphs which have a near perfect matching.
However, it is true if Conjecture \ref{FPMC} is true.  

\begin{conj}
Let $k \geq 3$ and $G$ be a $k$-critical graph. If $G$ does not have a 1-factor, then $\mu_f(G) > \mu(G)$. 
\end{conj}

\bibliographystyle{abbrv}

\bibliography{Literaturverzeichnis_Component_factors_of_critical_graphs}{}
\addcontentsline{toc}{section}{References}

\end{document}